\documentclass[reqno, 12pt]{amsart}

\pdfoutput=1

\usepackage{enumerate}
\usepackage{latexsym}
\usepackage[centertags]{amsmath}
\usepackage{amsfonts}
\usepackage{amssymb}
\usepackage{amsthm}
\usepackage{newlfont}
\usepackage{graphics}
\usepackage{color}
\usepackage{float}
\usepackage{diagbox}
\usepackage{hyperref}
\usepackage{longtable}
\usepackage{rotating}
\usepackage{multirow}

\usepackage{extarrows}

\usepackage[sort,compress,numbers]{natbib}

\usepackage[utf8]{inputenc}

\newtheorem{thm}{Theorem}[section]
\newtheorem{cor}[thm]{Corollary}
\newtheorem{lem}[thm]{Lemma}
\newtheorem{prop}[thm]{Proposition}

\newtheorem{cnj}[thm]{Conjecture}

\theoremstyle{mydefinition}
\newtheorem{dfn}[thm]{Definition}
\theoremstyle{myremark}

\newtheorem{exa}[thm]{Example}

\newtheorem{prob}[thm]{Open problem}
\allowdisplaybreaks[4]

\title{On Frobenius Numbers of Shifted Power Sequences}

\author{ Feihu Liu$^{1}$ and Guoce Xin$^{2, *}$}

\address{ $^{1, 2}$School of Mathematical Sciences,  Capital Normal University,
 Beijing 100048,  PR China}
\email{$^1$\texttt{liufeihu7476@163.com}\ \  \& $^2$\texttt{guoce\_xin@163.com}}
\date{\today}
\thanks{$*$ This work was partially supported by NSFC(12071311).}
\begin{document}
\maketitle

\begin{abstract}
We resolve the open problem of characterizing the Frobenius number $g(A)$ for shifted square sequences $A = (a, a+1^2, \ldots, a+k^2)$, confirming a conjecture of Einstein et al. (2007). By combining a combinatorial reduction to an optimization problem with Lagrange's Four-Square Theorem and generating function techniques, we derive an explicit formula for $g(A)$: a piecewise quadratic polynomial in $a$, classified by residue classes modulo $k^2$.
\end{abstract}

\def\D{{\mathcal{D}}}

\noindent
\begin{small}
 \emph{Mathematic subject classification}: Primary 11D07; Secondary 05A15, 11B75, 11D04.
\end{small}

\noindent
\begin{small}
\emph{Keywords}: Diophantine problem; Frobenius number; Four-square theorem; Generating function.
\end{small}

\section{Introduction}
Let $A = (a_1, a_2, \ldots, a_n)$ be a sequence of relatively prime positive integers, each at least $2$. The Frobenius number $g(A)$, defined as the largest integer not representable as a nonnegative integer linear combination of elements in $A$, has been extensively studied. For a comprehensive treatment, see \cite{Ramrez Alfonsn}. When $n=2$, Sylvester \cite{J. J. Sylvester,J. J. Sylvester1} established the formula $g(a_1,a_2) = a_1a_2 - a_1 - a_2$ in 1882. For $n=3$, a formula involving rational functions was introduced by Denham \cite{G.Denham}, and further investigated by Tripathi \cite{A. Tripathi1}. However, for $n \geq 4$, no general formula for $g(A)$ is known, though numerous special cases have been resolved (cf. \cite{M. Hujter,A. Brauer,Roberts1,E. S. Selmer,A. Tripathi2,A. L. Dulmage}). Computational approaches have been explored by Kannan \cite{R. Kannan} and Ram\'irez Alfons\'in \cite{J. L. Ramrez Alfonsn}.

Our primary objective is to resolve an open problem posed by Einstein, Lichtblau, Strzebonski, and Wagon \cite{D. Einstein}: characterizing $g(A)$ for the shifted square sequence $A = (a, a+1^2, a+2^2, \ldots, a+k^2)$. In 2007, these authors analyzed cases $k \in \{1,\dots,7\}$ using geometric algorithms and conjectured that $g(A)$ takes the form $\frac{1}{k^2}(a^2 + ca) - v$, where $c$ and $v$ are integers depending on $k$ and the residue class of $a$ modulo $k^2$, valid for sufficiently large $a$ (see \cite[Section 17]{D. Einstein}).

We prove this conjecture using the combinatorial method developed by Liu-Xin \cite{Liu-Xin}, leveraging Lagrange's Four-Square Theorem from number theory. To this end, we recall essential notation and results from \cite{Liu-Xin}. Throughout, $\mathbb{Z}$, $\mathbb{N}$, and $\mathbb{P}$ denote the set of integers, nonnegative integers, and positive integers, respectively.

A fundamental result of Brauer and Shockley is central to our approach:
\begin{thm}[\cite{J. E. Shockley}]
Let $A: = (a, B) = (a, b_1, \ldots, b_k)$ with $\gcd(A) = 1$, $d \in \mathbb{P}$, and $\gcd(a,d) = 1$. Define the set $\mathcal{R} = \left\{ ax + \sum_{i=1}^k b_i x_i \mid x, x_i \in \mathbb{N} \right\}$ and let $N_r := \min \{ a_0 \in \mathcal{R} \mid a_0 \equiv r \pmod{a} \}$. Then
$$g(A) = \max_{r \, \in \, \{0, \ldots, a-1\}} N_r - a = \max_{r \, \in \, \{0, \ldots, a-1\}} N_{dr} - a.$$
\end{thm}

For sequences of the form $A = (a, ha + dB) = (a, ha + d b_1, \ldots, ha + d b_k)$, computing $N_{dr}$ reduces to the minimization problem:
$$O_B(M) := \min \left\{ \sum_{i=1}^k x_i \,\middle|\, \sum_{i=1}^k b_i x_i = M,  x_i \in \mathbb{N} \right\}.$$
\begin{lem}[\cite{Liu-Xin}]\label{0202}
Let $A = (a, ha + d b_1, \ldots, ha + d b_k)$ with $k, h, d, b_i \in \mathbb{P}$ and $\gcd(A) = 1$. For $0 \leq r \leq a-1$,
\begin{equation}\label{0203}
N_{dr} = \min_{m \in \mathbb{N}} N_{dr}(m), \quad \text{where} \quad N_{dr}(m) := O_B(ma + r) \cdot ha + (ma + r)d.
\end{equation}
\end{lem}
This lemma extends to infinite sequences (see Lemma \ref{0506}).

The Four-Square Theorem becomes pivotal when considering the infinite shifted square sequence $A' = (a, a + B') = (a, a+1^2, a+2^2, \ldots)$. Here, the optimization problem $O_{B'}(M)$ admits an explicit solution via the Four-Square Theorem and related results, enabling us to determine $g(A')$. Similarly, we determine $g(A')$ for the shifted prime sequence $A' = (a, a+1, a+p_1, a+p_2, \ldots)$, where $(p_i)_{i\geq 1}$ denotes the prime sequence $(2, 3, 5, 7, \ldots)$.

The finite sequence $A = (a, a+1, a+2^2, \ldots, a+k^2)$ proves more challenging, as solving $O_B(M)$ directly is difficult. We overcome this by deriving bounds via the Four-Square Theorem, establishing stability properties using generating functions, and ultimately proving the conjecture on $g(A)$.

The paper is structured as follows. Sections 2 and 3 address the Frobenius number for the infinite sequences $A' = (a, a+1^2, a+2^2, \ldots)$ and $A' = (a, a+1, a+p_1, a+p_2, \ldots)$, respectively. For each, the associated problem $O_{B'}(M)$ is solved explicitly using number theory, yielding $g(A')$. Section 4 resolves the conjecture of Einstein et al. for the sequence $(a, a+1, a+2^2, \ldots, a+k^2)$. Concluding remarks are given in Section 5.

\section{Frobenius Number for Infinite Shifted Square Sequence}
The Frobenius number for sequences of shifted squares remains an open problem.
\begin{prob}{\em\cite[Problem 4]{D. Einstein}}\label{hahaha3}
    Given $k \in \mathbb{P}$ and $a > 2$, consider the sequence $A = (a, a+1, a+4, \ldots, a+k^2)$. How can we characterize $g(A)$?
\end{prob}

Einstein et al. \cite{D. Einstein} conjectured that for such sequences, the Frobenius number takes the form $\frac{1}{k^2}(a^2 + ca) - v$ for some integers $c$ and $v$, where these integers depend on $k$ and the residue class of $a$ modulo $k^2$. Through geometric methods, they verified this formula for $k \in \{1, 2, 3, 4, 5, 6, 7\}$ with corresponding $a \geq 1, 1, 16, 24, 41, 67, 136$ respectively.

In this section, we first address a simpler variant of this problem before examining the open problem in Section \ref{hahaha24}. Our primary focus is on the Frobenius number $g(A)$ for the infinite sequence obtained when $k \to \infty$, that is, $A = (a, a+1^2, a+2^2, a+3^2, \ldots)$. While sufficiently large $k$ satisfying $a+k^2 \ge g(a, a+1)$ become redundant in the finite case, their inclusion in the infinite setting simplifies our analysis.

To determine $g(A)$, we require the following concept.
\begin{dfn}\label{0504}
For $n \in \mathbb{P}$, define $\iota(n)$ as
\[\iota(n) = \min \left\{ s \mid n = a_1^2 + a_2^2 + \cdots + a_s^2,  a_i \in \mathbb{P}, 1 \leq i \leq s \right\}.\]
Thus, $\iota(n)$ denotes the minimal number of positive integer squares whose sum equals $n$.
\end{dfn}
Classical results in number theory characterize $\iota(n)$. Lagrange's Four-Square Theorem provides a fundamental bound:
\begin{lem}[Lagrange's Four-Square Theorem \cite{G. H. Hardy}]\label{0502}
Every positive integer can be expressed as the sum of four squares.
\end{lem}
To determine $\iota(n)$ precisely, we recall the standard prime factorization of $n \in \mathbb{P}$: $n = p_1^{e_1} p_2^{e_2} \cdots p_k^{e_k}$, where $p_i$ are distinct primes greater than 1 and $e_i > 0$ for $1 \leq i \leq k$.

Two additional lemmas refine this characterization:
\begin{lem}[\cite{G. H. Hardy}]\label{0501}
A positive integer $n$ is expressible as a sum of two squares if and only if, in its standard prime factorization, all prime factors congruent to $3 \pmod{4}$ have even exponents.
\end{lem}

\begin{lem}[\cite{G. H. Hardy}]\label{0503}
A positive integer $n$ is expressible as a sum of three squares if and only if $n$ cannot be written as $4^r(8t+7)$ for any $r, t \in \mathbb{N}$.
\end{lem}

Combining these results yields a complete classification:
\begin{thm}\label{0505}
Let $n \in \mathbb{P}$ with standard prime factorization $n = p_1^{e_1} p_2^{e_2} \cdots p_k^{e_k}$. Then $n$ belongs to exactly one of the following mutually exclusive categories:
\begin{enumerate}
\item[1)] All $e_i$ are even ($n$ is a perfect square).
\item[2)] At least one $e_i$ is odd, and for every prime $p_i \equiv 3 \pmod{4}$, the exponent $e_i$ is even.
\item[3)] $n = 4^r(8t+7)$ for some $r, t \in \mathbb{N}$.
\item[4)] None of the above conditions hold.
\end{enumerate}

Accordingly, $\iota(n)$ is determined by:
\[
\iota(n) =
\begin{cases}
1 & \text{if $n$ is of type 1)}, \\
2 & \text{if $n$ is of type 2)}, \\
4 & \text{if $n$ is of type 3)}, \\
3 & \text{if $n$ is of type 4)}.
\end{cases}
\]
\end{thm}
\begin{proof}
The result follows directly from Lemmas \ref{0502}, \ref{0501}, and \ref{0503}.
\end{proof}

To characterize \( g(A) \), we require the following lemma.
\begin{lem}\label{0506}
Let \( A = (a, a+1^2, a+2^2, a+3^2, \ldots) \). For \( 1 \leq r \leq a-1 \),
\[
N_r = \min \left\{ \iota(ma+r) \cdot a + ma + r \mid m \in \mathbb{N} \right\}.
\]
\end{lem}
\begin{proof}
We extend the definition of \( N_r \) appropriately:
\begin{align*}
N_r &= \min \left\{ \sum_{i=1}^{\infty} x_i (a + i^2) \mid \sum_{i=1}^{\infty} x_i (a + i^2) \equiv r \pmod{a},\ x_i \in \mathbb{N},\ i \geq 1 \right\} \\
&= \min \left\{ \left( \sum_{i=1}^{\infty} x_i \right) \cdot a + ma + r \mid \sum_{i=1}^{\infty} x_i \cdot i^2 = ma + r,\ m, x_i \in \mathbb{N},\ i \geq 1 \right\} \\
&= \min \left\{ \iota(ma + r) \cdot a + ma + r \mid m \in \mathbb{N} \right\}.
\end{align*}
(Note: Only finitely many \( x_i \) are non-zero.)
\end{proof}

\begin{thm}\label{0507}
Let \( A = (a, a+1^2, a+2^2, a+3^2, \ldots) \). Suppose there exists \( r \) with \( 1 \leq r \leq a-1 \) satisfying \( \iota(r) = 4 \), \( \iota(a+r) \geq 3 \), and \( \iota(2a+r) \geq 2 \). Then
\[
g(A) = 3a + \max \left\{ r \mid \iota(r) = 4,\  \iota(a+r) \geq 3,\  \iota(2a+r) \geq 2 \right\}.
\]
\end{thm}
\begin{proof}
The result follows directly from Lemma \ref{0506} and the definition \( g(A) = \max\{N_r\} - a \).
\end{proof}

When Theorem \ref{0507} is inapplicable, we obtain the following characterization.
\begin{thm}\label{0508}
Let \( A = (a, a+1^2, a+2^2, a+3^2, \ldots) \). For \( 1 \leq r \leq a-1 \), suppose the hypothesis of Theorem \ref{0507} fails, but at least one
of the following conditions holds:
\begin{enumerate}
    \item[(1)] \( \iota(r) = 4 \) and \( \iota(a+r) = 2 \);
    \item[(2)] \( \iota(r) = 4 \), \( \iota(a+r) \geq 3 \), and \( \iota(2a+r) = 1 \);
    \item[(3)] \( \iota(r) = 3 \) and \( \iota(a+r) \geq 2 \).
\end{enumerate}
Then
\[
g(A) = 2a + \max \left\{ r \mid r \text{ satisfies \emph{one} of conditions (1)--(3)} \right\}.
\]
\end{thm}
Table \ref{tab-B} in the appendix provides illustrative examples for Theorem \ref{0508}.

\begin{cnj}\label{0509}
Let \( A = (a, a+1^2, a+2^2, a+3^2, \ldots) \). For \( a > 30 \), Theorem \ref{0507} invariably applies, yielding
\[
g(A) = 3a + \max \left\{ r \mid \iota(r) = 4,\  \iota(a+r) \geq 3,\  \iota(2a+r) \geq 2 \right\}.
\]
\end{cnj}
Empirical evidence from Table \ref{tab-B} supports Conjecture \ref{0509}. Theorem \ref{0508} holds only for \( a \in \{4, 5, 7, 9, 10, 11, 13, 19, 21, 22, 30\} \). The minimal elements of the set \( \{4^r(8t+7) \mid r, t \in \mathbb{N}\} \) are \( 7, 15, 23, 28, 31, 39, 47, \ldots \). To disprove Conjecture \ref{0509}, one must find \( a > 30 \) such that for all \( n = 4^r(8t+7) < a \),
\(
\iota(a + n) \leq 2 \quad \text{or} \quad \iota(2a + n) = 1.
\)
We guess that it is impossible.

\section{Frobenius Number for Infinite Shifted Prime Sequence}
We now investigate the Frobenius number of the infinite shifted prime sequence. Denote the \( k \)-th prime by \( p_k \). The bound \( p_k \geq k \log k \) for \( k \geq 2 \) is classical \cite{J. B. Rosser,L. Schoenfeld}. Consider \( A = (a, a + p_0, a + p_1, \ldots, a + p_m, \ldots) \) where \( a \in \mathbb{P} \) (\( a > 2 \)), \( p_0 = 1 \), and \( p_1, p_2, \ldots \) is the prime sequence \( (2, 3, 5, 7, 11, \ldots) \). Our objective is to characterize \( g(A) \).

\begin{cnj}[\cite{P. Ribenboim}]\label{0601}
Every even integer greater than 2 equals the sum of two primes.
\end{cnj}
This is the \emph{strong Goldbach conjecture}. A related result, the \emph{weak Goldbach conjecture}, has been resolved by Helfgott.

\begin{prop}[\cite{H. A. Helfgott}]\label{0602}
Every odd integer greater than 7 equals the sum of three odd primes.
\end{prop}

\begin{dfn}\label{0603}
Let \( \mathcal{PN} \) denote the set of all prime numbers, i.e., \( \mathcal{PN} = \{2, 3, 5, 7, 11, \ldots\} \). For any \( n \in \mathbb{P} \), we define the symbol \( \tau(n) \) as
\[
\tau(n) = \min\left\{ s \,\bigg|\, n = p_1 + p_2 + \cdots + p_s, \ p_i \in \mathcal{PN} \cup \{1\}, \ 1 \leq i \leq s \right\}.
\]
It is evident that \( \tau(n) \leq 3 \).
\end{dfn}

\begin{thm}\label{0604}
Assuming the validity of the strong Goldbach conjecture (i.e., Conjecture~\ref{0601}), for any \( n \in \mathbb{P} \), we have
\[
\tau(n) =
\begin{cases}
1 & \text{if } n \in \mathcal{PN} \cup \{1\}, \\
3 & \text{if } n \text{ is odd, and } n, n-2 \notin \mathcal{PN} \cup \{1\}, \\
2 & \text{otherwise}.
\end{cases}
\]
\end{thm}
\begin{proof}
Under the assumption that Conjecture~\ref{0601} holds, this theorem follows directly from Proposition~\ref{0602} and Definition~\ref{0603}.
\end{proof}

For the sequence \( A = (a, a + p_0, a + p_1, \ldots, a + p_m, \ldots) \), we have

\begin{align*}
N_r&=\min \left\{ \sum_{i=0}^{\infty}x_i(a+p_i)\mid \sum_{i=0}^{\infty}x_i(a+p_i)\equiv r\pmod{a}, \ x_i\in \mathbb{N}, i\geq 0\right\}
\\&=\min \left\{ \left(\sum_{i=0}^{\infty}x_i\right)\cdot a+ma+r\mid \sum_{i=0}^{\infty}x_i\cdot p_i=ma+r, \ m, x_i\in \mathbb{N}, i\geq 0\right\}
\\&=\min \{\tau(ma+r)\cdot a+ma+r\mid m\in \mathbb{N}\}.
\end{align*}
Since only finitely many \( x_i \) are non-zero, we have \( g(A) = \max\{N_r\} - a \). By analogy with the shifted square sequence, we establish the following theorem.

\begin{thm}\label{0605}
Let \( A = (a, a + p_0, a + p_1, \ldots, a + p_m, \ldots) \), where \( a \in \mathbb{P} \), \( a > 2 \), \( p_0 = 1 \), and \( p_1, \ldots, p_m, \ldots \) form the sequence of primes. For \( 1 \leq r \leq a - 1 \), if there exists an \( r \) such that \( \tau(r) = 3 \) and \( \tau(a + r) \geq 2 \), then
\[
g(A) = 2a + \max\{ r \mid \tau(r) = 3, \ \tau(a + r) \geq 2 \}.
\]
\end{thm}

Similarly, we obtain the following result:

\begin{thm}\label{0606}
Let \( A = (a, a + p_0, a + p_1, \ldots, a + p_m, \ldots) \), where \( a \in \mathbb{P} \), \( a > 2 \), \( p_0 = 1 \), and \( p_1, \ldots, p_m, \ldots \) form the sequence of primes. For \( 1 \leq r \leq a - 1 \), if the condition in Theorem~\ref{0605} is not satisfied and at least one of the following holds:

1) There exists an \( r \) such that \( \tau(r) = 3 \) and \( \tau(a + r) = 1 \);

2) There exists an \( r \) such that \( \tau(r) = 2 \).

Then,
\[
g(A) = a + \max\{ r \mid \text{the } r \text{ in conditions 1) or 2)} \}.
\]
\end{thm}

To better understand the above theorem, we refer to Table~\ref{tab-D} in the appendix, which suggests the following result.

\begin{thm}\label{0607}
Let \( A = (a, a + p_0, a + p_1, \ldots, a + p_m, \ldots) \), where \( a \in \mathbb{P} \), \( a > 2 \), \( p_0 = 1 \), and \( p_1, \ldots, p_m, \ldots \) form the sequence of primes. If \( a > 44 \), then the formula \( g(A) \) in Theorem~\ref{0605} is the Frobenius number for the sequence \( A \), i.e., for any \( 1 \leq r \leq a - 1 \),
\[
g(A) = 2a + \max\{ r \mid \tau(r) = 3, \ \tau(a + r) \geq 2 \}.
\]
\end{thm}

\begin{proof}
Let \( \pi(x) \) denote the number of primes not exceeding \( x \). When \( a > 44 \), there exists \( r < a \) such that \( \tau(r) = 3 \). For example, \( r = 27 = 7 + 7 + 13 \) and \( r = 35 = 5 + 7 + 23 \). We must show that there exists an \( r < a \) such that \( \tau(r) = 3 \) and \( a + r \notin \mathcal{PN} \).

If \( a \) is odd, then \( a + r \) is even and hence not a prime. If \( a \) is even and sufficiently large, among the numbers less than \( a \): there are \( a/2 \) odd numbers (note that 1 is odd and 2 is prime); the number of \( r \) satisfying \( r \in \mathcal{PN} \) and \( r - 2 \in \mathcal{PN} \) is at most \( 2\pi(a) \); the number of \( r \) satisfying \( a + r \in \mathcal{PN} \) does not exceed \( \pi(2a) - \pi(a) \). Therefore, we need to prove
\[
\frac{a}{2} - 2\pi(a) - (\pi(2a) - \pi(a)) > 0.
\]

For a given \( a \), there exists \( k > 2 \) such that \( (k - 1)\log(k - 1) \leq a \leq k\log(k) \).
Using \( p_k \geq k\log(k) \), we need to prove
\[
\frac{(k - 1)\log(k - 1)}{2} - k - 2k > 0,
\]
which is equivalent to \( k > 411 \). Hence, the above inequality holds for \( a \geq 2467 \).
For even numbers \( a \) with \( 44 < a < 2467 \), we verified the theorem computationally. This completes the proof.
\end{proof}

\section{On Finite Shifted Square Sequence}\label{hahaha24}
In this section, we primarily address Open Problem \ref{hahaha3}. We aim to affirm the conjecture proposed by Einstein et al. \cite{D. Einstein} with the aid of the well-known ``Four-Square Theorem".

Our investigation commences with Lemma \ref{0202}, which says that we must determine the value of
$$N_r=\min\left\{O_B(ma+r) \cdot a+(ma+r) \mid \sum_{i=1}^kx_ii^2=ma+r, \ \ m,x_i\in \mathbb{N}, 1\leq i\leq k\right\}$$
for every $r$,
where
$$O_B(M)=\min\left\{\sum_{i=1}^kx_i \mid \sum_{i=1}^kx_ii^2=M, \ \ M,x_i\in \mathbb{N}, 1\leq i\leq k\right\}.$$
We introduce the function $\iota_k(M)=O_B(M)$ as defined below.

\begin{dfn}
For any integer $n \in \mathbb{P}$, we define $\iota_k(n)$ as
$$\iota_k(n) = \min\left\{ s \,\bigg|\, n = a_1^2 + a_2^2 + \cdots + a_s^2, \, a_i \in \{1, 2, 3, \ldots, k\}, \, 1 \leq i \leq s \right\}.$$
In other words, $\iota_k(n)$ represents the minimal number of squares required to express $n$ as a sum of squares from the set $\{1^2, 2^2, 3^2, \ldots, k^2\}$. A representation of $n$ using exactly $\iota_k(n)$ squares is called \emph{optimal}.
\end{dfn}

For example, if $k = 6$ and $n = 79$, then $\iota_k(79) = 4 = \iota(79)$ (as shown in Lemma \ref{0503}), with the optimal representation $n = 6^2 + 5^2 + 3^2 + 3^2$. It is worth noting that the greedy algorithm produces $n = 2 \cdot 6^2 + 2^2 + 3 \cdot 1^2$, which is not optimal.

Our goal is to develop efficient methods for computing $\iota_k(M)$. This will be addressed in the following subsection using generating functions.

\subsection{The Generating Function for $\iota_k(M)$}
Our primary objective in this subsection is to establish the following theorem.

\begin{thm}\label{t-OBM}
There exists a polynomial-time algorithm with respect to $k$ for computing $\iota_k(r)$ for all $r \in \mathbb{N}$.
\end{thm}

By definition,
\begin{equation}
\iota_k(r) = \min\left\{ \sum_{i=1}^k x_i \,\bigg|\, \sum_{i=1}^k x_i i^2 = r, \, r, x_i \in \mathbb{N}, \, 1 \leq i \leq k \right\}. \label{hahaha4}
\end{equation}
It is natural to consider the complete generating function:
$$
CF(t; q) = \prod_{i=1}^k \frac{1}{1 - t_i q^{i^2}} = \sum_{n \geq 0} \left( \sum_{ \substack{x_1 + 2^2 x_2 + \cdots + k^2 x_k = n} } t_1^{x_1} \cdots t_k^{x_k} \right) q^n,
$$
which encodes all nonnegative representations. By setting $t_i = t$ for all $i$, we obtain
$$
F(t, q) = \prod_{i=1}^k \frac{1}{1 - t q^{i^2}} = \sum_{n \geq 0} \left( \sum_{ \substack{x_1 + 2^2 x_2 + \cdots + k^2 x_k = n} } t^{x_1 + x_2 + \cdots + x_k} \right) q^n.
$$

If we denote a solution $(x_1, x_2, \ldots, x_k)$ for \eqref{hahaha4} satisfying $x_1 + \cdots + x_k = \iota_k(r)$ as \emph{optimal}, then the generating function
$$
f(t, q) = \sum_{n \geq 0} t^{\iota_k(n)} q^n
$$
extracts only one optimal representation, weighted by $t^{\iota_k(n)}$, for each $n$. It is straightforward to see that
$$
f(t, q) = \sum_{n \geq 0} t^{\iota_k(n)} q^n := \circledast F(t, q),
$$
where $\circledast$ is the operator defined as follows.

\begin{dfn}\label{hahaha12}
For a power series $G(t, q)$ in $t, q$ with nonnegative coefficients, we define $\circledast G(t, q)$ as the power series obtained from $G(t, q)$ by selecting the term of minimum degree (in $t$) within each coefficient (in $q$).
\end{dfn}

The key property $\circledast(\circledast F(t, q) \cdot \circledast G(t, q)) = \circledast (F(t, q) \cdot G(t, q))$ allows us to use Maple for efficient computation of the first $M + 1$ terms of $f(t, q)$ for any large $M$. The procedure is as follows:

\begin{enumerate}
  \item Initialize with $f_1:=\sum_{n= 0}^M t^n q^n$.

  \item Assuming $f_{i-1}$ has been calculated, proceed to compute $f_i$ as follows:
   first compute $\circledast f_{i-1} \cdot \sum_{n=0}^{h_i} t^n q^{i^2 n}$
  where $h_i= \lfloor M/i^2 \rfloor$ will be optimized,
  and then remove all terms with degrees in $q$ exceeding $M$.

  \item Set $f(t,q)=f_k$.
\end{enumerate}

Thus, we have established the following result.

\begin{lem}\label{hahaha10}
Let $k$ be a fixed positive integer. For a given $M$, the first $M + 1$ terms $f(t, q)|_{q^{\leq M}}$ of $f(t, q)$ can be computed in polynomial time in $M$. Consequently, $\iota_k(r)$ for $r \leq M$ can be computed in polynomial time in $M$.
\end{lem}

To completely determine $\iota_k(M)$ for any natural number $M$, we rely on two key inequalities derived from the ``Four-Square Theorem''.

\begin{prop}\label{hahaha1}
For any $M \in \mathbb{N}$, the following bounds hold:
 $\left\lceil \frac{M}{k^2}\right\rceil \leq \iota_k(M)\leq \left\lfloor \frac{M}{k^2}\right\rfloor+4$.
\end{prop}
\begin{proof}
Let $M$ be expressed as $M = s k^2 + r_1$, where $s \geq 0$ and $0 \leq r_1 \leq k^2 - 1$.
If $r_1 = 0$, then clearly $\iota_k(M) = s$.
If $1 \leq r_1 \leq k^2 - 1$, we have
$$
\left\lceil \frac{M}{k^2} \right\rceil = s + 1 \leq \iota_k(M) \leq s + \iota_k(r_1) \leq s + 4 = \left\lfloor \frac{M}{k^2} \right\rfloor + 4.
$$
The upper bound follows directly from the ``Four-Square Theorem''.
\end{proof}

The following lemma establishes that $\iota_k(r)$ exhibits a certain stability, thereby reducing the computation of $\iota_k(r)$ for $r \in \mathbb{N}$ to only a bounded range of $r$.

\begin{lem}\label{hahaha9}
For a given $k \in \mathbb{N}$, and for any $r \geq \left( \left\lceil \frac{3k}{2} \right\rceil - 2 \right) k^2$, with $r \in \mathbb{N}$, it holds that $\iota_k(k^2 + r) = \iota_k(r) + 1$.
\end{lem}

\begin{proof}
It is clear that $\iota_k(k^2 + r) \leq \iota_k(r) + 1$ always holds. If an optimal representation of $k^2 + r$ includes a term of $k^2$, then removing this term yields $\iota_k(r) \leq \iota_k(k^2 + r) - 1$, and thus $\iota_k(k^2 + r) = \iota_k(r) + 1$.

Assume, for contradiction, that $\iota_k(k^2 + r) < \iota_k(r) + 1$. This implies that $k^2 + r$ cannot have an optimal representation involving $k^2$. We will show that this leads to a contradiction.

Let $r = m k^2 + j$, where $0 \leq j \leq k^2 - 1$. By the Four-Square Theorem, $j$ can be expressed as $j = a^2 + b^2 + c^2 + d^2$, where $a, b, c, d \in \mathbb{N}$ and $a, b, c, d < k$.

Given the assumption for $r$, we have $m \geq \left\lceil \frac{3k}{2} \right\rceil - 2$. We choose $m$ such that $m k^2 \geq (m + 3)(k - 1)^2$. For such an $m$, the following inequality holds:
$$
m k^2 + j = m k^2 + a^2 + b^2 + c^2 + d^2 \geq (m + 3)(k - 1)^2.
$$
This implies that without using $k^2$, $\iota_k(k^2 + r)$ must be at least $m + 5$. Therefore, we have
$$
m + 5 \leq \iota_k(k^2 + r) < \iota_k(r) + 1 \leq m + 5,
$$
which is a contradiction.
\end{proof}

\begin{exa}\label{hahaha6}
If $k=3$ and $M=50$, then we have $h_2=12, h_3=5$, and $f_1=\sum_{n= 0}^{50} t^n q^n$.
Therefore, we compute
\begin{small}
\begin{align*}
f_2&=\circledast \left(f_1 \cdot \sum_{n=0}^{12} t^n q^{4 n}\right)\Bigg |_{q^{\leq 50}}
\\&={q}^{50}{t}^{14}+{q}^{49}{t}^{13}+{q}^{47}{t}^{14}+{q}^{48}{t}^{12}+{q
}^{46}{t}^{13}+{q}^{45}{t}^{12}+{q}^{43}{t}^{13}+{q}^{44}{t}^{11}+{q}^
{42}{t}^{12}+{q}^{41}{t}^{11}+{q}^{39}{t}^{12}
\\&\ \ +{q}^{40}{t}^{10}+{q}^{
38}{t}^{11}+{q}^{37}{t}^{10}+{q}^{35}{t}^{11}+{q}^{36}{t}^{9}+{q}^{34}
{t}^{10}+{q}^{33}{t}^{9}+{q}^{31}{t}^{10}+{q}^{32}{t}^{8}+{q}^{30}{t}^
{9}+{q}^{29}{t}^{8}
\\&\ \ +{q}^{27}{t}^{9}+{q}^{28}{t}^{7}+{q}^{26}{t}^{8}+{q
}^{25}{t}^{7}+{q}^{23}{t}^{8}+{q}^{24}{t}^{6}+{q}^{22}{t}^{7}+{q}^{21}
{t}^{6}+{q}^{19}{t}^{7}+{q}^{20}{t}^{5}+{q}^{18}{t}^{6}+{q}^{17}{t}^{5
}
\\&\ \ +{q}^{15}{t}^{6}+{q}^{16}{t}^{4}+{q}^{14}{t}^{5}+{q}^{13}{t}^{4}+{q}^
{11}{t}^{5}+{q}^{12}{t}^{3}+{q}^{10}{t}^{4}+{q}^{9}{t}^{3}+{q}^{7}{t}^
{4}+{q}^{8}{t}^{2}+{q}^{6}{t}^{3}+{q}^{5}{t}^{2}
\\&\ \ +{q}^{3}{t}^{3}+{q}^{4}t+{q}^{2}{t}^{2}+qt+1,
\end{align*}
\end{small}
and
\begin{small}
\begin{align*}
f(t, q)&=f_3=\circledast \left(f_2\cdot \sum_{n=0}^{5} t^n q^{9 n}\right)\Bigg |_{q^{\leq 50}}
\\&={q}^{50}{t}^{7}+{q}^{49}{t}^{6}+{q}^{48}{t}^{7}+{q}^{47}{t}^{7}+{q}^{
46}{t}^{6}+{q}^{45}{t}^{5}+{q}^{44}{t}^{6}+{q}^{43}{t}^{7}+{q}^{42}{t}
^{7}+{q}^{41}{t}^{6}+{q}^{40}{t}^{5}+{q}^{39}{t}^{6}
\\&\ \ +{q}^{38}{t}^{6}
+{q}^{37}{t}^{5}+{q}^{36}{t}^{4}+{q}^{35}{t}^{5}+{q}^{34}{t}^{6}+{q}^{33
}{t}^{6}+{q}^{32}{t}^{5}+{q}^{31}{t}^{4}+{q}^{30}{t}^{5}+{q}^{29}{t}^{
5}+{q}^{28}{t}^{4}+{q}^{27}{t}^{3}
\\&\ \ +{q}^{26}{t}^{4}+{q}^{25}{t}^{5}
+{q}^{24}{t}^{5}+{q}^{23}{t}^{4}+{q}^{22}{t}^{3}+{q}^{21}{t}^{4}+{q}^{20}{
t}^{4}+{q}^{19}{t}^{3}+{q}^{18}{t}^{2}+{q}^{17}{t}^{3}+{q}^{16}{t}^{4}
+{q}^{15}{t}^{4}
\\&\ \ +{q}^{14}{t}^{3}+{q}^{13}{t}^{2}+{q}^{12}{t}^{3}
+{q}^{11}{t}^{3}+{q}^{10}{t}^{2}+{q}^{7}{t}^{4}+{q}^{9}t+{q}^{8}{t}^{2}+{q}^
{6}{t}^{3}+{q}^{5}{t}^{2}+{q}^{3}{t}^{3}+{q}^{4}t
\\&\ \ +{q}^{2}{t}^{2}+qt+1.
\end{align*}
\end{small}
For instance, the coefficient $[q^{13}] f_2 =t^4$ while $[q^{13}] f_3=t^2$. These indicate: i) If using only $1,4$, the minimal number of squares needed to represent $13$ is $4$, achieved by $13 = 4 + 4 + 4 + 1$;
ii) If using $1,4,9$, the minimal number of squares needed to represent $13$ is $2$, achieved by $13 = 9 + 4$.

Lemma \ref{hahaha9} shows that it suffices to compute $\iota_k(r)$ for $r \leq u = \left(\left\lceil\frac{3k}{2}\right\rceil - 1\right)k^2$. In practice, this bound may be even smaller. For instance, when $k=3$, the bound suggests computing $\iota_3(r)$ for $r \leq 36$. Indeed, the stable property $\iota_3(r + 9) = \iota_3(r) + 1$ holds for all $r \geq 8$.

To better illustrate the stable property, consider the function $f(t, q) = \sum_{i=0}^8 f^i$, where $f^i$ extracts all terms corresponding to $q^{9s + i}$. The terms not implied by the stable property are highlighted in bold:

$$\begin{small}\begin{aligned}
f^0&=\mathbf{1}+tq^{9}+{t}^{2}{q}^{18}+{t}^{3}{q}^{27}+{t}^{4}{q}^{36}+{t}^{5}{q}^{45}\\
f^1&=\mathbf{tq}+{t}^{2}q^{10}+{t}^{3}{q}^{19}+{t}^{4}{q}^{28}+{t}^{5}{q}^{37}+{t}^{6}{q}^{46}\\
f^2&=\mathbf{{t}^{2}q^2}+{t}^{3}q^{11}+{t}^{4}{q}^{20}+{t}^{5}{q}^{29}+{t}^{6}{q}^{38}+{t}^{7}{
q}^{47}\\
f^3&=\mathbf{{t}^{3}q^3+{t}^{3}q^{12}}+{t}^{4}{q}^{21}+{t}^{5}{q}^{30}+{t}^{6}{q}^{39}+{t}^{7}{
q}^{48}\\
f^4&=\mathbf{tq^4}+{t}^{2}q^{13}+{t}^{3}{q}^{22}+{t}^{4}{q}^{31}+{t}^{5}{q}^{40}+{t}^{6}{q}^{49}\\
f^5&=\mathbf{{t}^{2}q^5}+{t}^{3}q^{14}+{t}^{4}{q}^{23}+{t}^{5}{q}^{32}+{t}^{6}{q}^{41}+{t}^{7}{
q}^{50}\\
f^6&=\mathbf{{t}^{3}q^6}+{t}^{4}q^{15}+{t}^{5}{q}^{24}+{t}^{6}{q}^{33}+{t}^{7}{q}^{42}\\
f^7&=\mathbf{{t}^{4}q^7+{t}^{4}q^{16}}+{t}^{5}{q}^{25}+{t}^{6}{q}^{34}+{t}^{7}{q}^{43}\\
f^8&=\mathbf{{t}^{2}q^8}+{t}^{3}q^{17}+{t}^{4}{q}^{26}+{t}^{5}{q}^{35}+{t}^{6}{q}^{44}.
\end{aligned}\end{small}$$
All values of $\iota_3(r)$ can be deduced from the boldfaced terms. For example, $\iota_3(52) = \iota_3(16 + 4 \times 9) = \iota_3(16) + 4 = 8$.
\end{exa}

Next, we demonstrate that $h_i$ can be much smaller through the following lemma.
\begin{lem}\label{hahaha7}
Let $(x_1,\dots, x_k)$ be an optimal solution for $\iota_k(r)$, where $r=\sum_{i=1}^k x_i i^2$. Then, the following inequalities hold:
$x_i\le 3$ for $i\le \left\lfloor \frac{k}{2}\right\rfloor$, and
$x_i \le \left\lfloor \frac{4k^2}{k^2-i^2}\right\rfloor $ for $\left\lfloor \frac{k}{2}\right\rfloor < i\leq k-1$.
\end{lem}
\begin{proof}
For $i \leq \left\lfloor \frac{k}{2} \right\rfloor$, since $4i^2$ can be expressed as $(2i)^2$, it follows that $x_i \leq 3$.

For $\left\lfloor \frac{k}{2} \right\rfloor < i \leq k - 1$, suppose to the contrary that $x_i > m_i := \left\lfloor \frac{4k^2}{k^2 - i^2} \right\rfloor$ in an optimal solution. Consider $\iota_k((m_i + 1)i^2)$. We have
\begin{small}
\begin{align*}
&m_i+1=\left\lfloor \frac{4k^2}{k^2-i^2}\right\rfloor+1>\frac{4k^2}{k^2-i^2}\Longleftrightarrow \frac{(m_i+1)(k^2-i^2)}{k^2}>4
\Longleftrightarrow \left\lceil \frac{(m_i+1)(k^2-i^2)}{k^2}\right\rceil >4
\\ \Longleftrightarrow & (m_i+1)-\left\lfloor \frac{(m_i+1)i^2}{k^2}\right\rfloor >4
 \Longleftrightarrow \left\lfloor \frac{(m_i+1)i^2}{k^2}\right\rfloor +4< m_i+1,
\end{align*}
\end{small}
This implies $\iota_k((m_i + 1)i^2) \leq \left\lfloor \frac{(m_i + 1)i^2}{k^2} \right\rfloor + 4 < m_i + 1$. Thus, we can replace $m_i + 1$ copies of $i^2$ with a better representation using only $\iota_k((m_i + 1)i^2)$ squares, leading to a contradiction.
\end{proof}

By Lemma \ref{hahaha7}, we can set $h_i = 3$ for $i \leq \left\lfloor \frac{k}{2} \right\rfloor$ and $h_i = \left\lfloor \frac{4k^2}{k^2 - i^2} \right\rfloor$ for $\left\lfloor \frac{k}{2} \right\rfloor < i \leq k - 1$. This significantly reduces the value of $h_i$. For example, when $k=3$, we find $h_1=3$ and $h_2=7$. Readers can compare this with Example \ref{hahaha6}. Note that in the first step, we should set $f_1 := \sum_{n=0}^3 t^n q^n$ (since $h_1=3$).

We conclude this subsection with the proof of Theorem \ref{t-OBM}.

\begin{proof}[Proof of Theorem \ref{t-OBM}]
By Lemma \ref{hahaha9}, when $r \geq \left(\left\lceil \frac{3k}{2} \right\rceil - 2\right)k^2$, we have $\iota_k(k^2 + r) = \iota_k(r) + 1$ and $\iota_k(sk^2 + r) = \iota_k(r) + s$ for $s \in \mathbb{N}$. Therefore, it suffices to determine $\iota_k(r)$ for $r \leq u = \left(\left\lceil \frac{3k}{2} \right\rceil - 1\right)k^2$. This allows us to impose the condition $x_k \leq \left\lceil \frac{3k}{2} \right\rceil - 1$ on the optimal representation. Thus, we set $h_k := \left\lceil \frac{3k}{2} \right\rceil - 1$ and $M := \left(\left\lceil \frac{3k}{2} \right\rceil - 1\right)k^2$.

By Lemma \ref{hahaha10}, $\iota_k(r)$ for $r \leq M$ can be computed in polynomial time in $M$. This completes the proof.
\end{proof}

\subsection{The Frobenius Number for Shifted Square Sequences}

First, we establish the following result regarding the behavior of $N_r$ when $a$ is sufficiently large.

\begin{lem}\label{hahaha2}
If $a \geq 3k^2$, then $N_r = N_r(0)$ for any given $r$, meaning $N_r$ attains its minimum value when $m = 0$.
\end{lem}
\begin{proof}
Fix $r$ and consider any non-negative integer $m$. Express $ma + r$ as $sk^2 + r_1$ where $0 \leq r_1 < k^2$. Given $a \geq 3k^2$, we have $(m+1)a + r = sk^2 + a + r_1 \geq (s+3)k^2 + r_1$. By Proposition~\ref{hahaha1}, it follows that $\iota_k(ma + r) = s + d$ and $\iota_k((m+1)a + r) \geq s + 3$ for some $0 \leq d \leq 4$. Therefore,
\begin{align*}
N_r(m+1) &\geq (s+3)a + (m+1)a + r = (s+4)a + ma + r \\
&\geq (s + d)a + ma + r = N_r(m).
\end{align*}
This shows that $N_r(m)$ is non-decreasing; hence, it attains its minimum at $m = 0$.
\end{proof}

\begin{lem}\label{hahaha11}
Let $k \geq 2$ and define $u = \left(\left\lceil\frac{3k}{2}\right\rceil + 1\right)k^2$. For any $a \geq u$, the following holds:
$$
\mathop{\max}\limits_{0 \leq r \leq a-1}\{N_r\} = \mathop{\max}\limits_{a - k^2 \leq r \leq a - 1}\{N_r\}.
$$
\end{lem}
\begin{proof}
Since $u \geq 3k^2$, by Lemma~\ref{hahaha2}, we have $N_r = N_r(0)$. From Lemma~\ref{hahaha9}, for $r \geq \left(\left\lceil\frac{3k}{2}\right\rceil - 2\right)k^2$, $\iota_k(sk^2 + r)$ is non-decreasing in $s \geq 0$. However, exceptions may occur for smaller $r$. For instance, when $k = 5$, we observe $\iota_k(7) = 4$, $\iota_k(5^2 + 7) = 2$, $\iota_k(2 \times 5^2 + 7) = 3$, etc.

Let $y = mk^2 + j$ with $1 \leq j \leq k^2 - 1$. The Four-Square Theorem ensures $\iota_k(y) \leq m + 4$. Additionally, we have $\iota_k(y + 3k^2) \geq \left\lceil \frac{y + 3k^2}{k^2} \right\rceil = m + 4$. Consequently, if $r \geq \left(\left\lceil\frac{3k}{2}\right\rceil - 2\right)k^2 + 2k^2$, then $\iota_k(r) \geq \iota_k(r - 3k^2)$. By the stable property, this implies $\iota_k(r) \geq \iota_k(r - sk^2)$ for all $s \geq 0$. Therefore, for $a \geq u = \left(\left\lceil\frac{3k}{2}\right\rceil - 2\right)k^2 + 3k^2$, we obtain $\mathop{\max}\limits_{0 \leq r \leq a - 1}\{N_r\} = \mathop{\max}\limits_{a - k^2 \leq r \leq a - 1}\{N_r\}$.
\end{proof}

The following theorem demonstrates that the Frobenius formula for shifted square sequences is a ``congruence class function'' modulo $k^2$, partitioning the function into $k^2$ classes based on the residue of $a$ modulo $k^2$. This structure is analogous to the conjecture proposed by Einstein et al.~\cite{D. Einstein}.

\begin{thm}\label{hahaha8}
Let $k$ be a fixed positive integer, and consider the sequence $A(a) = (a, a + 1, a + 2^2, \ldots, a + k^2)$. There exist non-decreasing sequences of non-negative integers $t_k = (t_{k,j})_{0 \leq j \leq k^2 - 1}$ and $r_k = (r_{k,j})_{0 \leq j \leq k^2 - 1}$ such that for all $a \geq u = \left(\left\lceil\frac{3k}{2}\right\rceil + 1\right)k^2$, the Frobenius number is given by:
$$
g(A(a)) = (t_{k,j} \cdot a + r_{k,j}) + (a + k^2)\left(\left\lfloor \frac{a}{k^2} \right\rfloor - \left\lceil \frac{3k}{2} \right\rceil - 1\right),
$$
where $a \equiv j \pmod{k^2}$. Furthermore, $t_{k,k^2 - 1} - t_{k,0} \leq 1$.
\end{thm}
\begin{proof}
Let $a = sk^2 + j \geq u = \left(\left\lceil\frac{3k}{2}\right\rceil + 1\right)k^2$, where $s \geq \left\lceil\frac{3k}{2}\right\rceil + 1$ and $0 \leq j \leq k^2 - 1$. By Lemma~\ref{hahaha11}, we have $\mathop{\max}\limits_{0 \leq r \leq a - 1}\{N_r\} = \mathop{\max}\limits_{a - k^2 \leq r \leq a - 1}\{N_r\}$.

Since $N_r = \iota_k(r) \cdot a + r$ is dominated by the coefficient $\iota_k(r)$, we first determine
$
\iota = \mathop{\max}\limits_{a - k^2 \leq r \leq a - 1} \iota_k(r)
$
and then identify the largest $a - k^2 \leq \widehat{r} \leq a - 1$ satisfying $\iota_k(\widehat{r}) = \iota$. Consequently, $\max\{N_r\} = \iota_k(\widehat{r}) \cdot a + \widehat{r}$, and the Frobenius number is $(\iota_k(\widehat{r}) - 1) \cdot a + \widehat{r}$.

For $u \leq a \leq u + k^2 - 1$ with $a \equiv j \pmod{k^2}$, $j$ ranges over $\{0, 1, 2, \ldots, k^2 - 1\}$. For each $j$, there exists $a - k^2 \leq r_{k,j} \leq a - 1$ such that $g(A) = \max\{N_r\} - a = t_{k,j} \cdot a + r_{k,j}$, where $t_{k,j}, r_{k,j} \in \mathbb{N}$ and $t_{k,j} = \iota_k(r_{k,j}) - 1$.

For $a = sk^2 + j$ with $s > \left\lceil\frac{3k}{2}\right\rceil + 1$ and $0\leq j\leq k^2-1$, we have $a = \left(s - \left\lceil\frac{3k}{2}\right\rceil - 1\right)k^2 + u + j$. By Lemma~\ref{hahaha9}, there exists $\overline{r} = \left(s - \left\lceil\frac{3k}{2}\right\rceil - 1\right)k^2 + r_{k,j}$ such that $\max\{N_r\} = N_{\overline{r}}$. Thus,
\begin{align*}
g(A(a)) &= \left(t_{k,j} + s - \left\lceil\frac{3k}{2}\right\rceil - 1\right) \cdot a + \left(s - \left\lceil\frac{3k}{2}\right\rceil - 1\right)k^2 + r_{k,j} \\
&= (t_{k,j} \cdot a + r_{k,j}) + (a + k^2)\left(\left\lfloor \frac{a}{k^2} \right\rfloor - \left\lceil \frac{3k}{2} \right\rceil - 1\right).
\end{align*}
This establishes the sequences $t_k$ and $r_k$ for each $0 \leq j \leq k^2 - 1$.

From the construction of $t_{k,j}$ and $r_{k,j}$, we observe that $u - k^2 \leq r_{k,j} < u + k^2$. For $a = u$, let $\widehat{r}$ be the largest integer satisfying
$$
u - k^2 \leq \widehat{r} \leq u - 1 \quad \text{and} \quad \iota_k(\widehat{r}) = \mathop{\max}\limits_{u - k^2 \leq r \leq u - 1} \iota_k(r).
$$
In the subsequent period from $u$ to $u+k^2-1$, corresponding to $a = u + k^2$, we have $\widehat{r}' = \widehat{r} + k^2$ satisfying
$$
u \leq \widehat{r}' \leq u + k^2 - 1 \quad \text{and} \quad \iota_k(\widehat{r}') = \iota_k(\widehat{r}) + 1 = \mathop{\max}\limits_{u \leq r \leq u + k^2 - 1} \iota_k(r).
$$
Within any interval of length $k^2$ starting with $a \equiv j \pmod{k^2}$ between $u - k^2$ and $u + k^2 - 1$, two cases arise:
\begin{enumerate}
\item If the interval contains $\widehat{r}'$, then $r_{k,j} = \widehat{r}'$;
\item If the interval contains $\widehat{r}$, then $r_{k,j}$ lies between $\widehat{r}$ and $\widehat{r}'$, with $\iota_k(r_{k,j})$ being either $\iota_k(\widehat{r})$ or $\iota_k(\widehat{r}) + 1$.
\end{enumerate}
Furthermore, $r_{k,j}$ is non-decreasing in $j$, implying that $t_k$ is non-decreasing and $t_{k,k^2 - 1} - t_{k,0} \leq 1$.
\end{proof}

\begin{cor}
When $a \geq u(k)$ (where $u(k)$ is a function dependent on $k \in \mathbb{P}$), the Frobenius formula for the sequence $A = (a, a + 1, a + 2^2, \ldots, a + k^2)$ can be viewed as a ``congruence class function'' with $k^2$ classes. Each segment of this function is a quadratic polynomial in $a$ with a leading coefficient $\frac{1}{k^2}$.
\end{cor}

Now we provide an example to illustrate the proof process intuitively.

\begin{exa}\label{hahaha22}
Let $k=3$. Our bound for $a$ is $a \geq u = 54$, and we know that $45 \leq r_{k,j} < 63$. By referring to Example \ref{hahaha6}, we have
\begin{align*}
\circledast F(t,q) = \cdots &+ t^5 q^{45} + t^6 q^{46} + t^7 q^{47} + t^7 q^{48} + t^6 q^{49} + t^7 q^{50} + t^8 q^{51} + \underline{t^8 q^{52}} + t^7 q^{53} \\
&+ t^6 q^{54} + t^7 q^{55} + \mathbf{t^8 q^{56}} + \mathbf{t^8 q^{57}} + t^7 q^{58} + \mathbf{t^8 q^{59}} + \mathbf{t^9 q^{60}} + \underline{t^9 q^{61}} + t^8 q^{62} + \cdots,
\end{align*}
where we have presented only the necessary part of our proof: i) The first row corresponds to $u-k^2\le r<u-1$;
ii) The second row corresponds to $u\le r \le u+k^2-1$; iii) Multiplying the first row by $tq^9$ yields the second row.

The underlined terms correspond to $\widehat{r}$ and $\widehat{r}'$, with
$$
\iota_k(\widehat{r}) = \max_{45 \leq r \leq 53} \iota_k(r) = \iota_k(52) = 8 \quad \text{and} \quad \iota_k(\widehat{r}') = 9.
$$
We now explain how to determine $r_{3,j}$ from the above expression. For $a = u = 54$, we find $r_{3,0} = 52$, which corresponds to $\widehat{r}$. Consequently,
$t_{3,0} = \iota_3(52) - 1 = 7$, leading to $g(A) = 7a + 52 = 430$.
Furthermore, if $a \equiv 0 \pmod{k^2}$ and $a \geq 54$, we deduce that
$$
g(A(a)) = (7a + 52) + (a + 9)\left(\left\lfloor \frac{a}{9} \right\rfloor - 6\right).
$$

For $a = 55$, since $\iota_k(54) = 6 < \iota_k(52) = 8$, we obtain $r_{3,1} = 52$. Similarly, for $a = 56$, we have $r_{3,2} = 52$. However, for $a = 57$, the situation differs, and we find $r_{3,3} = 56$. This is because we encounter the boldfaced term, which corresponds to $8 = \iota_k(56) \geq \iota_k(52) = 8$.

The boldfaced terms represent left-to-right maximums with respect to the power of $t$. Through analogous reasoning, we obtain $r_{3,4} = r_{3,5} = 57$, $r_{3,6} = 59$, $r_{3,7} = 60$, and $r_{3,8} = 61$.

In summary, we have $t_{3} = [7, 7, 7, 7, 7, 7, 7, 8, 8]$ and $r_{3} = [52, 52, 52, 56, 57, 57, 59, 60, 61]$. These values allow us to construct the following Frobenius formula for $k = 3$:
$$\begin{aligned}
  g(A(a))=
\left\{
    \begin{array}{lc}
         (7a+52)+(a+9)(\lfloor \frac{a}{9}\rfloor -6) &\ \text{if}\ \ a\equiv 0,1,2 \mod 3^2;\\
         (7a+56)+(a+9)(\lfloor \frac{a}{9}\rfloor -6) & \text{if}\ \ a\equiv 3 \mod 3^2;\ \ \ \ \  \\
         (7a+57)+(a+9)(\lfloor \frac{a}{9}\rfloor -6) & \text{if}\ \ a\equiv 4,5 \mod 3^2;\ \ \\
         (7a+59)+(a+9)(\lfloor \frac{a}{9}\rfloor -6) & \text{if}\ \ a\equiv 6 \mod 3^2;\ \ \ \ \  \\
         (8a+60)+(a+9)(\lfloor \frac{a}{9}\rfloor -6) & \text{if}\ \ a\equiv 7 \mod 3^2;\ \ \ \ \ \\
         (8a+61)+(a+9)(\lfloor \frac{a}{9}\rfloor -6) & \text{if}\ \ a\equiv 8 \mod 3^2.\ \ \ \ \
    \end{array}
\right.
\end{aligned}$$

Similarly, for $k = 1, 2, 4, 5$, we obtain the following results:
$$\begin{small}\begin{aligned}
\bullet\ \ \ t_{1}&=[1], r_{1}=[2].\\
\bullet\ \ \ t_{2}&=[5,5,5,5], r_{2}=[15,15,15,18].\\
\bullet\ \ \ t_{4}&=[8,8,8,8,8,8,8,8,9,9,9,9,9,9,9,9],\\
r_{4}&=[101,101,101,101,115,115,117,118,119,119,119,119,119,124,124,126].\\
\bullet\ \ \ t_{5}&=[10,10,10,10,11,11,11,11,11,11,11,11,11,11,11,11,11,11,11,11,11,11,11,11,11],\\
r_{5}&=[224,224,224,227,228,228,228,231,231,231,231,231,231,237,237,237,240,240,240,\\
&\ \ \ \ \ 240,244,244,246,247,247].
\end{aligned}\end{small}$$
One can verify that our results are consistent with those of Einstein et al. \cite{D. Einstein}.
\qed
\end{exa}

Our bound $a \geq u = \left(\left\lceil \frac{3k}{2} \right\rceil + 1\right)k^2$ in Lemma \ref{hahaha9} and \ref{hahaha11} is not tight. For example, when $k = 3$, the above formula holds for $a \geq 16$. For a specific $k$, by computing $\iota_k(r)$ and analyzing the ``stabilization" process, we can determine the precise lower bound of $a$ such that the Frobenius formula in Theorem \ref{hahaha8} remains valid. Using Maple, we obtain the following result.

\begin{cor}\label{hahaha5}
Let the exact lower bound of $a$ be denoted by $\widehat{u}$. Then we have
\begin{align*}
k &= (1, 2, 3, 4, 5, 6, 7, 8, 9, 10, 11, 12, 13, \ldots) \\
\widehat{u} &= (1, 1, 16, 24, 41, 68, 137, 168, 379, 558, 451, 709, 987, \ldots).
\end{align*}
\end{cor}

We have addressed Open Problem \ref{hahaha3} concerning the Frobenius number of a shifted square sequence. Based on the above theorems and corollaries, our findings reveal both similarities and contrasts with the conjecture proposed by Einstein et al.

1. Our approach integrates the Four-Square Theorem with the optimization problem $\iota_k(M)$ to derive a general formula for the Frobenius number for any $k$. In contrast, Einstein et al. employed their geometric algorithm to obtain explicit formulas for specific cases where $k \in \{1,2,3,4,5,6,7\}$.

2. Our derived formulas are consistent with those of Einstein et al. The fundamental structure is a ``congruence class function'' partitioned into $k^2$ distinct classes. Each class is characterized by a quadratic polynomial in $a$, with the leading coefficient $\frac{1}{k^2}$.

3. Our proof establishes a general lower bound for $a$ as $a \geq u = (\lceil\frac{3k}{2}\rceil + 1)k^2$. In comparison, Einstein et al. provided specific bounds of $(1,1,16,24,41,67,136)$ for $k = (1,2,3,4,5,6,7)$, respectively. For any given $k$, our method allows precise determination of the lower bound via $\iota_k(r)$, with computational assistance. We have tabulated these bounds for $k \leq 13$ in Corollary \ref{hahaha5}. Notably, for $k=6,7$, Einstein et al.'s results state $a \geq 67, 136$, whereas our findings yield $a \geq 68, 137$. This discrepancy appears to be a typographical error in \cite{D. Einstein}.

4. The proof of Theorem \ref{hahaha8} introduces a systematic method for computing the coefficients $t_{k,j}$ and $r_{k,j}$ through the evaluation of $\iota_k(r)$. We have demonstrated this methodology with explicit coefficients for $1 \leq k \leq 5$ in Example \ref{hahaha22}. For higher values of $k$, interested readers are encouraged to perform similar computations.

\section{Concluding Remark}
Our primary contribution is Theorem \ref{hahaha8}, which resolves the conjecture proposed by Einstein et al. regarding the Frobenius number of a shifted square sequence. This methodology can be extended to address shifted high power sequences by leveraging results from \emph{Waring's Problem}. Waring's Problem seeks the smallest integer $m = \kappa(n)$ such that any positive integer $a$ can be expressed as the sum of $m$ $n$th powers of nonnegative integers. For a comprehensive discussion, refer to \cite[Chapters XX--XXI]{G. H. Hardy} and the references therein.

We present the following established results and conjectures related to $\kappa(n)$.

\begin{thm}[\cite{G. H. Hardy}]
The values of $\kappa(n)$ are known for small exponents: $\kappa(2)=4$, $\kappa(3)=9$, $\kappa(4)=19$, and $\kappa(5)=37$.
\end{thm}

\begin{cnj}[\cite{G. H. Hardy}]
For Waring's Problem, the following formula for $\kappa(n)$ is conjectured:
$$\kappa(n) = 2^n + \left\lfloor \frac{3^n}{2^n} \right\rfloor - 2.$$
This conjecture has been computationally verified for $6 \leq n < 471600000$.
\end{cnj}

Extending our previous analysis, for a shifted high power sequence $A = (a, a+1, a+2^n, \ldots, a+k^n)$, when $a \geq u(k)$ (where $u(k)$ is a function dependent on $k \in \mathbb{P}$), the Frobenius formula exhibits a ``congruence class function'' structure with $k^n$ classes. Each class is represented by a quadratic polynomial in $a$, with a leading coefficient of $\frac{1}{k^n}$. Notably, when $n=1$ and $k \leq a-1$, this reduces to Brauer's result as presented in \cite{A. Brauer}.

One of our ongoing research directions is to generalize the Frobenius formula to sequences of the form $A = (a, ha + db_1, ha + db_2, \ldots, ha + db_k) = (a, ha + dB)$. Initial findings can be found in \cite{LiuXinYeYin2024}.

\noindent\textbf{Statements and Declarations:} The authors declared that they had no conflicts of interest with respect to their authorship or the publication of this article.

\noindent\textbf{Acknowledgements:}
The authors would like to express their sincere appreciation for all suggestions for improving the presentation of this paper.
This work was partially supported by the National Natural Science Foundation of China [12071311].

\appendix
\section{Tables \ref{tab-B} and \ref{tab-D}}
\vspace{-5mm}
\begin{tiny}
\begin{table}[htbp]
    	\centering
    	\caption{Discussion on $A=(a, a+1^2, a+2^2, a+3^2, a+4^2, \ldots)$}\label{tab-B}
    	\begin{tabular}{c||l|c|c||c||l|c|c}
    		\hline \hline
    		Value of $a$ & The $r$ of $\max\{N_r\}$ & $g(A)$ & Theorem & Value of $a$ & The $r$ of $\max\{N_r\}$ & $g(A)$ & Theorem  \\
    		\hline
    		$2$ & $1$ & $1$ & no & $23$ & $15=3^2+2^2+1+1$ & $84$ & \text{Theorem \ref{0507}}  \\
    		\hline
    		$3$ & $2=1+1$ & $5$ & no & $24$ & $23=3^2+3^2+2^2+1$ & $95$ & \text{Theorem \ref{0507}} \\
    		\hline
    		$4$ & $3=1+1+1$ & $11$ & \textcolor{blue}{Theorem \ref{0508}} & $25$ & $23=3^2+3^2+2^2+1$ & $98$ & \text{Theorem \ref{0507}}  \\
    		\hline
            $5$ & $3=1+1+1$ & $13$ & \textcolor{blue}{Theorem \ref{0508}} & $26$ & $7=2^2+1+1+1$ & $85$ & \text{Theorem \ref{0507}}  \\
    		\hline
            $6$ & $5=2^2+1$ & $11$ & no & $27$ & $15=3^2+2^2+1+1$ & $96$ & \text{Theorem \ref{0507}}  \\
    		\hline
            $7$ & $6=2^2+1+1$ & $20$ & \textcolor{blue}{Theorem \ref{0508}} & $28$ & $23=3^2+3^2+2^2+1$ & $107$ & \text{Theorem \ref{0507}}   \\
    		\hline
            $8$ & $7=2^2+1+1+1$ & $31$ & \text{Theorem \ref{0507}} & $29$ & $28=5^2+1+1+1$ & $115$ & \text{Theorem \ref{0507}}  \\
    		\hline
     $9$ & $6=2^2+1+1$ & $24$ & \textcolor{blue}{Theorem \ref{0508}} & $30$ & $28=5^2+1+1+1$ & $88$ & \textcolor{blue}{Theorem \ref{0508}} \\
    		\hline
     $10$ & $7=2^2+1+1+1$ & $27$ & \textcolor{blue}{Theorem \ref{0508}} & $31$ & $28=5^2+1+1+1$ & $121$ & \text{Theorem \ref{0507}} \\
    		\hline
    $11$ & $7=2^2+1+1+1$ & $29$ & \textcolor{blue}{Theorem \ref{0508}} & $32$ & $31=5^2+2^2+1+1$ & $127$ & \text{Theorem \ref{0507}} \\
    		\hline
    $12$ & $7=2^2+1+1+1$ & $43$ & \text{Theorem \ref{0507}} & $33$ & $23=3^2+3^2+2^2+1$ & $122$ & \text{Theorem \ref{0507}}  \\
    		\hline
    $13$ & $11=3^2+1+1$ & $37$ &  \textcolor{blue}{Theorem \ref{0508}} & $34$ & $28=5^2+1+1+1$ & $130$ & \text{Theorem \ref{0507}}  \\
    		\hline
    $14$ & $7=2^2+1+1+1$ & $49$ & \text{Theorem \ref{0507}} & $35$ & $31=5^2+2^2+1+1$ & $136$ & \text{Theorem \ref{0507}} \\
    		\hline
    $15$ & $7=2^2+1+1+1$ & $52$ & \text{Theorem \ref{0507}} & $36$ & $31=5^2+2^2+1+1$ & $139$ & \text{Theorem \ref{0507}}   \\
    		\hline
    $16$ & $15=3^2+2^2+1+1$ & $63$ & \text{Theorem \ref{0507}} & $37$ & $23=3^2+3^2+2^2+1$ & $134$ & \text{Theorem \ref{0507}}  \\
    		\hline
    $17$ & $7=2^2+1+1+1$ & $58$ & \text{Theorem \ref{0507}} & $38$ & $31=5^2+2^2+1+1$ & $145$ & \text{Theorem \ref{0507}}  \\
    		\hline
    $18$ & $15=3^2+2^2+1+1$ & $69$ & \text{Theorem \ref{0507}} & $39$ & $31=5^2+2^2+1+1$ & $148$ & \text{Theorem \ref{0507}}  \\
    		\hline
    $19$ & $15=3^2+2^2+1+1$ & $53$ & \textcolor{blue}{Theorem \ref{0508}} & $40$ & $39=5^2+3^2+2^2+1$ & $159$ & \text{Theorem \ref{0507}} \\
    		\hline
    $20$ & $15=3^2+2^2+1+1$ & $75$ & \text{Theorem \ref{0507}} & $41$ & $28=5^2+1+1+1$ & $151$ & \text{Theorem \ref{0507}} \\
    		\hline
    $21$ & $19=3^2+3^2+1$ & $61$ & \textcolor{blue}{Theorem \ref{0508}} & $42$ & $28=5^2+1+1+1$ & $154$ & \text{Theorem \ref{0507}} \\
    		\hline
    $22$ & $21=4^2+2^2+1$ & $65$ & \textcolor{blue}{Theorem \ref{0508}} & $\cdots$ & $\cdots$ & $\cdots$ & $\cdots$  \\
    		\hline
    	\end{tabular}
    \end{table}
\end{tiny}
\begin{tiny}
\begin{table}[htbp]
    	\centering
    	\caption{Discussion on $A=(a, a+1, a+2, a+3, a+5, a+7, a+11,\ldots)$}\label{tab-D}
    	\begin{tabular}{c||l|c|c||c||l|c|c}
    		\hline \hline
    		Value of $a$ & The $r$ of $\max\{N_r\}$ & $g(A)$ & Theorem &Value of $a$ & The $r$ of $\max\{N_r\}$ & $g(A)$ & Theorem \\
    		\hline
    		$2$ & $1$ & $1$ & no  & $31$ & $27=7+7+13$ & $89$ & \text{Theorem \ref{0605}}\\
    		\hline
    		$3$ & $2$ & $2$ & no & $32$ & $30=7+23$ & $62$ &  \textcolor{blue}{Theorem \ref{0606}} \\
    		\hline
    		$4$ & $3$ & $3$ & no & $33$ & $27=7+7+13$ & $93$ & \text{Theorem \ref{0605}} \\
    		\hline
            $5$ & $4=2+2$ & $9$ & \textcolor{blue}{Theorem \ref{0606}}& $34$ & $33=2+31$ & $67$ &  \textcolor{blue}{Theorem \ref{0606}} \\
    		\hline
            $6$ & $4=2+2$ & $10$ & \textcolor{blue}{Theorem \ref{0606}} & $35$ & $27=7+7+13$ & $97$ & \text{Theorem \ref{0605}}\\
    		\hline
            $7$ & $6=3+3$ & $13$ & \textcolor{blue}{Theorem \ref{0606}} & $36$ & $27=7+7+13$ & $99$ & \text{Theorem \ref{0605}} \\
    		\hline
            $8$ & $6=3+3$ & $14$ & \textcolor{blue}{Theorem \ref{0606}} & $37$ & $35=5+7+23$ & $109$ & \text{Theorem \ref{0605}} \\
    		\hline
            $9$ & $8=3+5$ & $17$ & \textcolor{blue}{Theorem \ref{0606}} & $38$ & $27=7+7+13$ & $103$ & \text{Theorem \ref{0605}} \\
    		\hline
            $10$ & $9=2+7$ & $19$ & \textcolor{blue}{Theorem \ref{0606}} & $39$ & $35=5+7+23$ & $113$ & \text{Theorem \ref{0605}} \\
    		\hline
            $11$ & $10=3+7$ & $21$ & \textcolor{blue}{Theorem \ref{0606}} & $40$ & $35=5+7+23$ & $115$ & \text{Theorem \ref{0605}} \\
    		\hline
            $12$ & $10=3+7$ & $22$ & \textcolor{blue}{Theorem \ref{0606}} & $41$ & $35=5+7+23$ & $117$ & \text{Theorem \ref{0605}} \\
    		\hline
            $13$ & $12=5+7$ & $25$ & \textcolor{blue}{Theorem \ref{0606}} & $42$ & $35=5+7+23$ & $119$ & \text{Theorem \ref{0605}} \\
    		\hline
            $14$ & $12=5+7$ & $26$ & \textcolor{blue}{Theorem \ref{0606}} & $43$ & $35=5+7+23$ & $121$ & \text{Theorem \ref{0605}} \\
    		\hline
            $15$ & $14=7+7$ & $29$ & \textcolor{blue}{Theorem \ref{0606}} & $44$ & $42=5+37$ & $86$ &  \textcolor{blue}{Theorem \ref{0606}}\\
    		\hline
            $16$ & $15=2+13$ & $31$ & \textcolor{blue}{Theorem \ref{0606}} & $45$ & $35=5+7+23$ & $125$ & \text{Theorem \ref{0605}}  \\
    		\hline
            $17$ & $16=3+13$ & $33$ & \textcolor{blue}{Theorem \ref{0606}} & $46$ & $35=5+7+23$ & $127$ & \text{Theorem \ref{0605}}  \\
    		\hline
           $18$ & $16=3+13$ & $34$ & \textcolor{blue}{Theorem \ref{0606}} & $47$ & $35=5+7+23$ & $129$ & \text{Theorem \ref{0605}} \\
    		\hline
           $19$ & $18=5+13$ & $37$ & \textcolor{blue}{Theorem \ref{0606}} & $48$ & $27=7+7+13$ & $123$ & \text{Theorem \ref{0605}} \\
    		\hline
           $20$ & $18=5+13$ & $38$ & \textcolor{blue}{Theorem \ref{0606}} & $49$ & $35=5+7+23$ & $133$ & \text{Theorem \ref{0605}} \\
    		\hline
           $21$ & $20=7+13$ & $41$ & \textcolor{blue}{Theorem \ref{0606}} & $50$ & $35=5+7+23$ & $135$ & \text{Theorem \ref{0605}}  \\
    		\hline
          $22$ & $21=2+19$ & $43$ & \textcolor{blue}{Theorem \ref{0606}} &$51$ & $35=5+7+23$ & $137$ & \text{Theorem \ref{0605}} \\
    		\hline
           $23$ & $22=3+19$ & $45$ & \textcolor{blue}{Theorem \ref{0606}} & $52$ & $35=5+7+23$ & $139$ & \text{Theorem \ref{0605}} \\
    		\hline
          $24$ & $22=3+19$ & $46$ & \textcolor{blue}{Theorem \ref{0606}} & $53$ & $51=3+7+41$ & $157$ & \text{Theorem \ref{0605}} \\
    		\hline
          $25$ & $24=5+19$ & $49$ & \textcolor{blue}{Theorem \ref{0606}} & $54$ & $51=3+7+41$ & $159$ & \text{Theorem \ref{0605}} \\
    		\hline
          $26$ & $25=2+23$ & $51$ & \textcolor{blue}{Theorem \ref{0606}} & $55$ & $51=3+7+41$ & $161$ & \text{Theorem \ref{0605}} \\
    		\hline
            $27$ & $26=3+23$ & $53$ & \textcolor{blue}{Theorem \ref{0606}} & $56$ & $35=5+7+23$ & $147$ & \text{Theorem \ref{0605}} \\
    		\hline
     $28$ & $27=7+7+13$ & $83$ &  \text{Theorem \ref{0605}} & $57$ & $51=3+7+41$ & $165$ & \text{Theorem \ref{0605}}\\
    		\hline
     $29$ & $27=7+7+13$ & $85$ & \text{Theorem \ref{0605}} & $58$ & $57=3+7+47$ & $173$ & \text{Theorem \ref{0605}} \\
    		\hline
    $30$ & $27=7+7+13$ & $87$ & \text{Theorem \ref{0605}} & $\cdots$ & $\cdots$ & $\cdots$ & $\cdots$ \\
    		\hline
    	\end{tabular}
    \end{table}
\end{tiny}

\end{document}